\documentclass[english]{article}
\usepackage[T1]{fontenc}
\usepackage[latin9]{inputenc}
\usepackage{geometry}
\usepackage{mathrsfs}
\usepackage{mathtools}
\usepackage{amsmath}
\usepackage{amsthm}
\usepackage{amssymb}
\usepackage{stackrel}
\usepackage{setspace}

\makeatletter
\theoremstyle{plain}
\newtheorem{thm}{\protect\theoremname}[section]
\theoremstyle{definition}
\newtheorem{defn}[thm]{\protect\definitionname}
\theoremstyle{plain}
\newtheorem{cor}[thm]{\protect\corollaryname}
\theoremstyle{remark}
\newtheorem{rem}[thm]{\protect\remarkname}
\theoremstyle{plain}
\newtheorem{prop}[thm]{\protect\propositionname}
\ifx\proof\undefined
\newenvironment{proof}[1][\protect\proofname]{\par
	\normalfont\topsep6\p@\@plus6\p@\relax
	\trivlist
	\itemindent\parindent
	\item[\hskip\labelsep\scshape #1]\ignorespaces
}{%
	\endtrivlist\@endpefalse
}
\providecommand{\proofname}{Proof}
\fi
\theoremstyle{plain}
\newtheorem{lem}[thm]{\protect\lemmaname}

\@ifundefined{date}{}{\date{}}
\makeatother

\usepackage{babel}
\providecommand{\corollaryname}{Corollary}
\providecommand{\definitionname}{Definition}
\providecommand{\lemmaname}{Lemma}
\providecommand{\propositionname}{Proposition}
\providecommand{\remarkname}{Remark}
\providecommand{\theoremname}{Theorem}

\begin{document}
\title{LOCAL CONDITIONAL REGULARITY FOR THE LANDAU EQUATION WITH COULOMB
POTENTIAL}
\maketitle
\begin{center}
Immanuel Ben Porat\footnote{CMLS, Ecole polytechnique, 91128 Palaiseau cedex, France 

immanuel.ben-porath@polytechnique.edu }
\par\end{center}
\begin{abstract}
This paper studies the regularity of Villani solutions of the space
homogeneous Landau equation with Coulomb interaction in dimension
3. Specifically, we prove that any such solution belonging to the
Lebesgue space $L_{t}^{\infty}L_{v}^{q}$ with $q>3$ in an open cylinder
$(0,S)\times B$, where $B$ is an open ball of $\mathbb{R}^{3}$,
must have H\"older continuous second order derivatives in the velocity
variables, and first order derivative in the time variable locally
in any compact subset of that cylinder. 
\end{abstract}

\section{Introduction }

\begin{onehalfspace}
The objective of this work is to derive conditional regularity of
certain weak solutions $f=f(t,v)\geq0$ a.e. of the space homogeneous
Landau equation with Coulomb potential. At the formal level, this
equation reads (with the Einstein summation convention being used) 

\begin{equation}
\partial_{t}f-\overline{a_{ij}}^{f}\frac{\partial f}{\partial v_{i}\partial v_{j}}=8\pi f^{2},\label{eq:1-1}
\end{equation}

where 
\[
\overline{a_{ij}}^{f}(t,v)=a_{ij}\star f(t,\cdot),
\]
 with 
\[
a_{ij}(z)=\frac{1}{|z|}(\delta_{ij}-\frac{z_{i}z_{j}}{|z|^{2}}).
\]

A fundamental existence theorem due to Villani (theorem 3, (i) in
\cite{20}) provides the existence of a special class of weak solutions
of the associated Cauchy problem for (\ref{eq:1-1}). For an earlier
approach to the existence of weak solutions see \cite{2}. These solutions
are known as ``H- solutions'' or ``Villani solutions'' (see definition
(\ref{Definition 2.1})). In this work we will be concerned only with
this class of solutions. 

We note the similarity between equation (\ref{eq:1-1}) and the semilinear
heat equation 
\begin{equation}
u_{t}-\Delta u=u^{2},\label{eq:-4}
\end{equation}
 for which finite time blow up may occur (see e.g. theorem 1 in \cite{21}).
The term $8\pi f^{2}$ on the r.h.s. of (\ref{eq:1-1}) clearly promotes
finite time blow-up. On the other hand, if $f$ increases at some
point, the diffusion matrix $\overline{a_{ij}}^{f}$ in (\ref{eq:1-1})
increases as well. This is an important difference between the Landau
equation (\ref{eq:1-1}) and the semilinear heat equation \ref{eq:-4}.
Since any increase in the diffusion matrix offsets the effect of the
quadratic source term $8\pi f^{2}$, whether $f$ blows up in finite
time remains an major open problem at the time of this writing. 

One of the most major recent contributions in the study of equation
(\ref{eq:1-1}) was produced in \cite{4}, where it is proved that
for all $T>0$ any H-solution $f$ of (\ref{eq:1-1}) with finite
mass, energy and entropy must satisfy the bound 
\[
\stackrel[0]{T}{\int}\underset{\mathbb{R}^{3}}{\int}\frac{|\nabla_{v}\sqrt{f}|^{2}}{(1+|v|^{2})^{\frac{3}{2}}}dvdt<\infty.
\]
Observe that the bound above offers a striking similarity with Leray's
theory of weak solutions to the Navier-Stokes equation in space dimension
$3$. Indeed, $\sqrt{f}$ belongs to $L^{\infty}((0,+\infty);L^{2}(\mathbf{\mathbb{R}}^{3}))$
while $\nabla_{v}\sqrt{f}$ belongs to $L^{2}((0,+\infty);L_{loc}^{2}(\mathbf{\mathbb{R}}^{3}))$.
Apart from the $(1+|v|^{2})^{-\frac{3}{2}}$ weight in the dissipation
estimate, these bounds on $\sqrt{f}$ are reminiscent of the bounds
satisfied by Leray weak solutions of the Navier-Stokes equation in
space dimension $3$. This loose analogy suggests the problem of proving
conditional regularity results for H-solutions of the Landau equation
(\ref{eq:1-1}). For instance, one could seek conditional, local regularity
results analogous to those obtained in the work of Serrin {[}17{]}
for the Navier-Stokes equation. Namely, Serrin proves that weak solutions
of the Navier Stokes equations lying in an appropiate mixed Lebesgue
space are smooth with respect to the spatial variable and locally
H\"older continuous with respect to time. In view of the analogy
observed above, it is reasonable to expect that an analogous result
holds for equation (\ref{eq:1-1}). 

Silvestre's result in \cite{18} can be thought of as a global Serrin
type theorem for the Landau equation (\ref{eq:1-1}). This result
is complemented in some sense by the propagation estimates obtained
in \cite{1}, which allow to get improved regularity on the solution
by imposing integrability and suitable smallness assumptions on the
initial data. Other conditional regularity results include e.g. \cite{12},\cite{13},
\cite{10} and \cite{11}. In \cite{12} the authors prove regularity
of radially symmetric $L^{p}$ solutions with $p>\frac{3}{2}$ . In
\cite{11}, local H\"older continuity is proved for essentially bounded
weak solutions of equation (\ref{eq:1-1}). In \cite{13} regularity
is proved provided the solutions satisfy a certain variant of Poincare
inequality together with what is referred by the authors as the ``local
doubling property'', which is a reminiscent of the weak Harnack inequality
for supersolutions to parabolic equations. Finally, in \cite{10}
regularity is proved provided the solutions satisfy a variant of the
entropy inequality, which is a reminiscent of the Leray energy inequality,
satisfied by some suitable weak solutions Navier Stokes equations.
Regularity is of course related to uniqueness, and so we refer the
reader to \cite{8} for relevant uniqueness results. 

Our new contribution here is obtaining a local Serrin type theorem
for equation (\ref{eq:1-1}), that is, a localisation of Silvestre's
theorem. This is formulated precisely in theorem \ref{Theorem 1.1 },
which is the main result of the paper. Our approach here differs from
the strategy introduced in \cite{18}, which is based on apriori estimates
for the Landau equation. It is not clear whether it is possible to
establish a local version of these apriori estimates, and we shall
not pursue this direction here. Instead, our approach relies crucially
on analyzing the regularity of the coefficients $\overline{a_{ij}}^{f}$,
and can be roughly summarized as follows: 

1. Obtaining regularity on the coefficients $\overline{a_{ij}}^{f}$(this
is the content of lemma \ref{Lemma 3.6}).

2. Once regularity on the coefficients is achieved, we may apply classical
existence and uniqueness results from the theory of linear parabolic
PDE's in order to show that the solution $f$ enjoys 2 weak derivatives
in space and one weak derivative in time (this is the content of theorem
\ref{Theorem 3.4}).

3. Once the solution is known to be (locally) in some appropriate
space-time Sobolev space, we may improve the regularity of the solution
via a bootstrap argument.
\end{onehalfspace}
\begin{onehalfspace}

\section{Preliminaries and Main Results }
\end{onehalfspace}
\begin{onehalfspace}

\subsection{General Notations}
\end{onehalfspace}

\begin{onehalfspace}
We shall consider here space dimension $3$. Let us fix some notation
and recall some basic definitions. We shall denote by $\omega\subset(0,\infty)\times\mathbb{R}^{3}$
some cylinder, that is, a set of the form $J\times B$ where $J=(0,S)\subset(0,\infty)$
is some finite open interval and $B\subset\mathbb{R}^{3}$ is some
open ball. Of course, upon shifting the spatial variable, it is sufficient
to prove theorem \ref{Theorem 1.1 } for $B$ centered at the origin.
Let $1\leq p\leq\infty,1\leq q\leq\infty,k\in\mathbb{R}$. We recall
the following function spaces. 

$L_{t}^{p}L_{v}^{q}(\omega)\coloneqq L^{p}L^{q}(\omega):$ The space
of all measurable maps $f:\omega\rightarrow\mathbb{R}$ with $||t\mapsto||f(t,\cdot)||_{q}||_{p}<\infty$. 

$L_{k}^{p}(\mathbb{R}^{3}):$ the space of all measurable maps $f:\mathbb{R}^{3}\rightarrow\mathbb{R}$
with $||v\mapsto(1+|v|^{2})^{\frac{k}{2}}f||_{p}<\infty$.

$L\log L(\mathbb{R}^{3})\coloneqq\{f\in L^{1}(\mathbb{R}^{3})|\underset{\mathbb{R}^{3}}{\int}|f(v)||\log|f(v)||dv<\infty\}$.

$W_{q}^{2,1}(\omega):$ the Banach space consisting of all elements
in $L^{q}(\omega)$ with generalized derivatives of the form $\partial_{t}^{r}\partial_{x}^{s}$
where $2r+s\leq2$ and such that $\partial_{t}^{r}\partial_{x}^{s}\in L^{q}(\omega)$.
The norm on $W_{q}^{2,1}(\omega)$ is defined by $||u||_{q}^{(2)}=\stackrel[j=0]{2}{\sum}\underset{2r+s=j}{\sum}||\partial_{t}^{r}\partial_{x}^{s}u||_{q}$. 

$W_{2}^{1,0}(\omega):$ the Hilbert space consisting of all elements
in $L^{2}(\omega)$ with generalized derivatives of the form $\partial_{x}$
such that $\partial_{x}\in L^{2}(\omega)$. The scalar product on
$W_{2}^{1,0}(\omega)$ is defined by $(u,v)=\underset{\omega}{\int}uv+\partial_{x_{k}}u\partial_{x_{k}}vdxdt$. 

Let $0<l<1$. We will use the following H\"older spaces. 

$H_{l}^{\ast}(\overline{\omega})$: for each $P,Q\in\omega$ we introduce
the metric $|P-Q|=\max\{|x^{P}-x^{Q}|,8|t^{P}-t^{Q}|^{\frac{1}{2}}\}$.
Write $P\backsim Q$ to mean $P=Q+\eta e_{k}$ for some $\eta\in\mathbb{R},1\leq k\leq3$
(as customary $e_{k}$ stands for the unit vector in direction $k$
in $\mathbb{R}^{3}$). We consider the distances $d_{P}=\inf\{|P-Q|\}_{Q\in T(P)}$
where $T(P)$ is the set of points on the boundary of $\omega$ for
which there exist a continuous arc connecting $P,Q$ along which the
$t$ coordinate is nondecreasing from $Q$ to $P.$ This gives rise
to a distance $d_{PQ}$ defined by $d_{PQ}=\min\{d_{P},d_{Q}\}$.
For $m\in\mathbb{N}$ we then consider 

\[
||d^{m}v||_{l}^{\ast}=\underset{P\in\omega}{\sup}d_{P}^{m}|v(P)|+\underset{P,Q\in\omega,P\thicksim Q}{\sup}d_{PQ}^{m+l}\frac{|v(P)-v(Q)|}{|P-Q|^{l}}.
\]

$(H_{l}^{\ast}(\overline{\omega}),||d\cdot||_{l}^{\ast})$ is the
Banach space whose elements are all $v$ on $\omega$ admiting a (unique)
$C^{0}$ extension to $\overline{\omega}$ and such that $||dv||_{l}^{\ast}<\infty$
.

$H_{l+2}^{\ast}(\overline{\omega}):$ we consider the norm 

\[
||v||_{l+2}^{\ast}=||v||_{l}+\stackrel[i=1]{3}{\sum}||d\partial_{x_{i}}v||_{l}+\underset{1\leq i,j\leq3}{\sum}||d^{2}\partial_{x_{i}x_{j}}v||_{l}+||d^{2}\partial_{t}v||_{l}^{\ast},
\]

where 

\[
||d^{m}v||_{l}=\underset{P\in\omega}{\sup}d_{P}^{m}|v(P)|+\underset{P,Q\in\omega}{\sup}d_{PQ}^{m+l}\frac{|v(P)-v(Q)|}{|P-Q|^{l}},m\in\mathbb{N}.
\]

$(H_{l+2}^{\ast}(\overline{\omega}),||\cdot||_{l+2}^{\ast})$ is the
Banach space whose elements are all $v$ on $\omega$ admiting a (unique)
$C^{2}$ extension to $\overline{\omega}$ and such that $||v||_{l+2}^{\ast}<\infty$. 

$H^{l,\frac{l}{2}}(\overline{\omega}):$ we consider the norm 

\[
||v||^{l,\frac{l}{2}}=\underset{P\in\omega}{\sup}|v(P)|+\underset{(t,x),(t,y)\in\omega}{\sup}\frac{|v(t,x)-v(t,y)|}{|x-y|^{l}}+\underset{(t,x),(s,x)\in\omega}{\sup}\frac{|v(t,x)-v(s,x)|}{|t-s|^{\frac{l}{2}}}.
\]

$(H^{l,\frac{l}{2}}(\overline{\omega}),||\cdot||^{l,\frac{l}{2}})$
is the Banach space whose elements are all $v$ on $\omega$ admiting
a (unique) $C^{0}$ extension to $\overline{\omega}$ and such that
$||v||^{l,\frac{l}{2}}<\infty$ .

$H^{1+l,\frac{1+l}{2}}(\overline{\omega}):$ we consider the norm 

\[
||v||^{1+l,\frac{1+l}{2}}=\underset{P\in\omega}{\sup}|v(P)|+\underset{(t,x),(t,y)\in\omega}{\sup}\frac{|v(t,x)-v(t,y)|}{|x-y|^{l}}+\underset{(t,x),(s,x)\in\omega}{\sup}\frac{|v(t,x)-v(s,x)|}{|t-s|^{\frac{l}{2}}}+
\]

\[
\stackrel[i=1]{3}{\sum}\underset{P\in\omega}{\sup}|\partial_{x_{i}}v(P)|+\stackrel[i=1]{3}{\sum}\underset{(t,x),(t,y)\in\omega}{\sup}\frac{|\partial_{x_{i}}v(t,x)-\partial_{x_{i}}v(t,y)|}{|x-y|^{l}}+\stackrel[i=1]{3}{\sum}\underset{(t,x),(s,x)\in\omega}{\sup}\frac{|\partial_{x_{i}}v(t,x)-\partial_{x_{i}}v(s,x)|}{|t-s|^{\frac{l}{2}}}.
\]

We remark that the notations used for the above H\"older type spaces
has nothing to do with Sobolev spaces (which are frequently denoted
by $H^{k}$). 
\end{onehalfspace}
\begin{onehalfspace}

\subsection{The Landau Equation and H-Solutions}
\end{onehalfspace}

\begin{onehalfspace}
Following \cite{20}, let us briefly recall the Landau equation and
the notion of H-solutions (also know as Villani solutions). We refer
the reader to \cite{20} for a elaborative and motivational disscusion. 

Define $\Pi:\mathbb{R}^{3}\setminus\{0\}\rightarrow M_{3}(\mathbb{R})$
by $\Pi(z)=I-(\frac{z}{|z|})^{\otimes2}$, so that $\Pi_{ij}(z)=\delta_{ij}-\frac{z_{i}z_{j}}{|z|^{2}}$.
Define $a_{ij}(z)=\frac{1}{|z|}\Pi_{ij}(z)$ and $b_{i}(z)=\stackrel[j=1]{3}{\sum}\partial_{j}a_{ij}(z)$.
For a function $f$ on $\mathbb{R}^{3}$ consider the convolutions
$\overline{a}_{ij}^{f}\coloneqq\underset{\mathbb{R}^{3}}{\int}a_{ij}(v-z)f(z)dz,\overline{b}_{i}^{f}\coloneqq\underset{\mathbb{R}^{3}}{\int}b_{i}(v-z)f(z)dz$.
As is customary, for a function $f$ on $\mathbb{R}^{3}$ we shall
write $M(f)=\underset{\mathbb{R}^{3}}{\int}f(v)dv,E(f)=\underset{\mathbb{R}^{3}}{\int}f(v)\frac{|v|^{2}}{2}dv,H(f)=\underset{\mathbb{R}^{3}}{\int}f(v)\log(f(v))dv$
whenever the quantity on the RHS is well defined. The quantities $M(f),E(f),H(f)$
are called the mass, energy and entropy of $f$ respectively. 
\end{onehalfspace}
\begin{defn}
\begin{onehalfspace}
\begin{flushleft}
\label{Definition 2.1} Let $f_{0}(v)=f_{0}$ have finite mass, energy
and entropy. A H-solution to equation (\ref{eq:1-1}) with initial
data $f_{0}$ on $[0,T]\times\mathbb{R}^{3}$ is an element $f\in C([0,T],\mathcal{D}'(\mathbb{R}^{3}))\cap L^{1}((0,T),L_{-1}^{1}(\mathbb{R}^{3}))$
such that
\par\end{flushleft}
\begin{flushleft}
1. $f\geq0$ and $\forall t\in[0,T]:f(t,\cdot)\in L_{2}^{1}(\mathbb{R}^{3})\cap L\log L(\mathbb{R}^{3})$ 
\par\end{flushleft}
\begin{flushleft}
2. $f(0,\cdot)=f_{0}(\cdot)$
\par\end{flushleft}
\begin{flushleft}
3. $\forall t\in[0,T]$: $\underset{\mathbb{R}^{3}}{\int}f(t,v)\log(f(t,v))dv\leq\underset{\mathbb{R}^{3}}{\int}f_{0}(v)\log(f_{0}(v))dv$
and $\underset{\mathbb{R}^{3}}{\int}f(t,v)\psi(v)dv=\underset{\mathbb{R}^{3}}{\int}f_{0}(v)\psi(v)dv$
where $\psi=1,v_{i},|v|^{2}$
\par\end{flushleft}
\begin{flushleft}
4. $\forall\varphi\in C^{1}([0,T],C_{0}^{\infty}(\mathbb{R}^{3})),\forall t\in[0,T]$
\[
\underset{\mathbb{R}^{3}}{\int}f(t,v)\varphi(t,v)dv-\underset{\mathbb{R}^{3}}{\int}f_{0}(v)\varphi(0,v)dv-\stackrel[0]{t}{\int}\underset{\mathbb{R}^{3}}{\int}f(\tau,v)\partial_{t}\varphi(\tau,v)dvd\tau
\]
\par\end{flushleft}
\begin{flushleft}
\[
=-\stackrel[0]{t}{\int}\underset{\mathbb{R}^{3}}{\int}\underset{\mathbb{R}^{3}}{\int}\Pi(\nabla_{v}\sqrt{\frac{f(\tau,v)f(\tau,w)}{|v-w|}}-\nabla_{w}\sqrt{\frac{f(\tau,v)f(\tau,w)}{|v-w|}})\sqrt{\frac{f(\tau,v)f(\tau,w)}{|v-w|}})(\nabla_{v}\varphi(t,v)-\nabla_{w}\varphi(t,w))dvdwd\tau.
\]
\par\end{flushleft}
\end{onehalfspace}

\end{defn}
\begin{onehalfspace}
The integral on the RHS in 4. is well defined, as explained in detail
in \cite{20}. Apriori, it is not clear what is the relation between
the notion of a H-solution and a weak solution in the classical sense.
As we will recall in the next section, it can be shown that in fact
H-solutions are weak solutions in the classical sense, but this is
a consequence of a highly nontrivial theorem of Desvillettes \cite{4}.
In what follows, we shall refer to H-solutions of equation (\ref{eq:1-1})
simply as H-solutions. 
\end{onehalfspace}
\begin{onehalfspace}

\subsection{Known and New Results }
\end{onehalfspace}

\begin{onehalfspace}
First and most foremost we recall that Villani proved the global existence
of H-solutions for all initial data with finite mass, energy and entropy.
In addition he proved that these solutions are weakly H\"older continuous
in time, as described in the following 
\end{onehalfspace}
\begin{thm}
\begin{onehalfspace}
\begin{flushleft}
\label{Theorem 2.2} \textup{(\cite{20}, Theorem 3, (i))} Let $f_{0}:\mathbb{R}^{3}\rightarrow\mathbb{R}$
have finite mass, energy and entropy. Then there exist a H-solution
$f$ with initial datum $f_{0}$. Moreover, if $f$ is a H-solution
with initial datum $f_{0}$, then, for all $\varphi\in W^{2,\infty}(\mathbb{R}^{3})$,
$t\mapsto\underset{\mathbb{R}^{3}}{\int}f(t,v)\varphi(v)dv$ is H\"older
continuous with exponent $\frac{1}{2}$. 
\par\end{flushleft}
\end{onehalfspace}

\end{thm}
\begin{onehalfspace}
The next most important result for our purposes is the following weighted
$L^{2}$ estimate on the distributional derivative of $\sqrt{f}$ 
\end{onehalfspace}
\begin{thm}
\begin{onehalfspace}
\begin{flushleft}
\label{Theorem 2.3 } \textup{(\cite{4}, Theorem 1)} Let $f$ be
a H-solution on $[0,T]\times\mathbb{R}^{3}$. Then $\stackrel[0]{T}{\int}\underset{\mathbb{R}^{3}}{\int}\frac{|\nabla_{v}\sqrt{f(t,v)}|^{2}}{(1+|v|^{2})^{\frac{3}{2}}}dvdt<\infty$.
\par\end{flushleft}
\end{onehalfspace}

\end{thm}
\begin{onehalfspace}
We remark that theorem \ref{Theorem 2.3 } implies in particular that
any $L^{2}(\omega)$ H-solution is in $W_{2}^{1,0}(\omega)$, as can
be seen from the identity $\partial_{v_{j}}f=2\sqrt{f}\partial_{v_{j}}\sqrt{f}$.
A list of important conclusions is derived from theorem \ref{Theorem 2.3 }
in \cite{4}. The first one, is that H-solutions are in fact ``usual''
weak solutions in the sense of integration against test functions
\end{onehalfspace}
\begin{cor}
\begin{onehalfspace}
\begin{flushleft}
\label{Corollary 2.3} \textup{(\cite{4}, Corollary 1.1)} Let $f$
be a H-solution with initial datum $f_{0}$. Then $f\in L^{1}((0,T),L_{-3}^{3}(\mathbb{R}^{3}))$
and for all $\varphi\in C_{0}^{2}([0,T)\times\mathbb{R}^{3})$ it
holds that 
\par\end{flushleft}
\begin{flushleft}
\[
-\underset{\mathbb{R}^{3}}{\int}f_{0}(v)\varphi(0,v)dv-\stackrel[0]{T}{\int}\underset{\mathbb{R}^{3}}{\int}f(t,v)\partial_{t}\varphi(t,v)dvdt=
\]
\par\end{flushleft}
\begin{flushleft}
\[
\frac{1}{2}\stackrel[i=1]{3}{\sum}\stackrel[j=1]{3}{\sum}\stackrel[0]{T}{\int}\underset{\mathbb{R}^{3}}{\int}\underset{\mathbb{R}^{3}}{\int}f(t,v)f(t,w)a_{ij}(v-w)(\partial_{ij}\varphi(t,v)+\partial_{ij}\varphi(t,w))dvdwdt
\]
\[
+\stackrel[i=1]{3}{\sum}\stackrel[0]{T}{\int}\underset{\mathbb{R}^{3}}{\int}\underset{\mathbb{R}^{3}}{\int}f(t,v)f(t,w)b_{i}(v-w)(\partial_{i}\varphi(t,v)-\partial_{i}\varphi(t,w))dvdwdt.
\]
\par\end{flushleft}
\end{onehalfspace}

\end{cor}
\begin{onehalfspace}

\end{onehalfspace}\begin{rem}
\begin{onehalfspace}
\begin{flushleft}
\label{Remark 2.4} If $f$ satisfies the condition $||f(t,\cdot)||_{L^{q}(B)}\leq S_{0}$
for some $S_{0}>0,q>3$ then for any $\varphi\in C_{0}^{\infty}((0,S)\times B)$
it holds that 
\par\end{flushleft}
\begin{flushleft}
\begin{equation}
-\stackrel[0]{T}{\int}\underset{\mathbb{R}^{3}}{\int}f\partial_{t}\varphi+\stackrel[0]{T}{\int}\underset{\mathbb{R}^{3}}{\int}\overline{a}_{ij}^{f}\partial_{i}f\partial_{j}\varphi=\stackrel[0]{T}{\int}\underset{\mathbb{R}^{3}}{\int}f\overline{b}_{i}^{f}\partial_{i}\varphi(t,v).\label{eq:-3}
\end{equation}
\par\end{flushleft}
\end{onehalfspace}
\begin{onehalfspace}
As we will see in lemma (\ref{Lemma 3.6}), subject to the above condition
$\overline{a}_{ij}^{f},\overline{b}_{i}^{f}\in L_{\mathrm{loc}}^{\infty}(\omega)$,
which together with theorem (\ref{Theorem 2.3 }), justifies that
the above integrals are well defined. 
\end{onehalfspace}
\end{rem}
\begin{onehalfspace}
Corollary \ref{Corollary 2.3} and the above remark will enable us
to apply classical regularity, uniqueness and existence results from
the theory of linear parabolic PDEs, which have the above weak formulation.
We finish this section by stating the main result of this work 
\end{onehalfspace}
\begin{thm}
\begin{onehalfspace}
\begin{flushleft}
\label{Theorem 1.1 } Let $f$ be a H-solution on $[0,T]\times\mathbb{R}^{3}$.
Let $\omega=(0,S)\times B$ be an open cylinder in $[0,T]\times\mathbb{R}^{3}$.
Suppose there exist $S_{0}>0,q>3$ such that for all $t\in(0,S)$
one has $||f(t,\cdot)||_{L^{q}(B)}\leq S_{0}$ . Then for all $0<\alpha<1$
we have $f\in H_{\alpha+2}^{\ast}(\overline{\Omega})$ for each $\Omega\Subset\omega$. 
\par\end{flushleft}
\end{onehalfspace}
\end{thm}
\begin{rem}
\begin{onehalfspace}
\begin{flushleft}
We point out that theorem \ref{Theorem 1.1 } implies that $f$ is
in fact a classical solution to equation (\ref{eq:1-1}) in $\omega$.
Indeed, from the last assertion of theorem \ref{Theorem 3.7} we know
that $f$ is a strong solution to equation \ref{eq:1-1}, while theorem
\ref{Theorem 1.1 } in particular implies that the weak time derivative
and the first and second order weak spatial derivatives are continuous. 
\par\end{flushleft}
\end{onehalfspace}
\end{rem}
\begin{onehalfspace}

\section{Conditional Regularity}
\end{onehalfspace}
\begin{onehalfspace}

\subsection{From H-Solutions To $\mathbf{\mathbf{W}_{2}^{2,1}}$}
\end{onehalfspace}

\begin{onehalfspace}
This subsection is the first step towards the conditional local regularity
of H-solutions as stated in Theorem \ref{Theorem 1.1 }. We recall
that we have adapted the notation $\omega=J\times B=(0,S)\times B.$
We prove 
\end{onehalfspace}
\begin{thm}
\begin{onehalfspace}
\begin{flushleft}
\label{Theorem 3.4} Let $f$ be a H-solution. Suppose there exist
$S_{0}>0$ and $q>3$ such that for all $t\in J$ one has $||f(t,\cdot)||_{L^{q}(B)}\leq S_{0}$
. Then $f\in W_{2}^{2,1}(\Omega)$ for all $\Omega\Subset\omega$.
Moreover, $f$ is a strong solution to equation (\ref{eq:1-1}) in
$\omega$, that is 
\[
\partial_{t}f-\overline{a_{ij}}^{f}\frac{\partial f}{\partial v_{i}\partial v_{j}}=8\pi f^{2}
\]
 for a.e. $(t,x)\in\omega$. 
\par\end{flushleft}
\end{onehalfspace}

\end{thm}
\begin{onehalfspace}
First, we recall that local coercivity for the coefficients $\overline{a}_{ij}^{f}$
has been established in proposition 2.3 of \cite{1}. A straightforward
conclusion of the latter is local ellipticity of the coefficients,
as summarized in the following 
\end{onehalfspace}
\begin{prop}
\begin{onehalfspace}
\begin{flushleft}
\label{Proposition 2.2  } \textup{(Local Ellipticity)} There are
constants $0<c=c(M_{0},E_{0},H_{0},K),C=C(M_{0},S_{0},q)$ with the
following property. Suppose $f\in L_{2}^{1}(\mathbb{R}^{3})\cap L\log L(\mathbb{R}^{3})$
satisfy $f\geq0$ a.e. and $M(f)=M_{0},E(f)\leq E_{0},H(f)\leq H_{0},||f||_{L^{q}(B)}\leq S_{0}$
for some $\frac{3}{2}<q\leq\infty$. Let $K\Subset\mathbb{R}^{3}$.
Then: $\forall\xi\in\mathbb{R}^{n},v\in K:c|\xi|^{2}\leq\overline{a_{ij}}^{f}(v)$$\xi_{i}\xi_{j}\leq C|\xi|^{2}$.
\par\end{flushleft}
\end{onehalfspace}
\end{prop}
\begin{proof}
\begin{onehalfspace}
The coercivity estimate $c|\xi|^{2}\leq\overline{a_{ij}}^{f}(v)\xi_{i}\xi_{j}$
was established in proposition 2.3 of \cite{1}. In addition for all
$\xi\in S^{1}$ we have 

\[
|\overline{a}_{ij}^{f}(v)\xi_{i}\xi_{j}|\leq\underset{\mathbb{R}^{3}}{\int}\frac{2}{|z|}|f(v-z)|dz=\underset{B}{\int}\frac{2}{|z|}|f(v-z)|dz+\underset{\mathbb{R}^{3}-B}{\int}\frac{2}{|z|}|f(v-z)|dz\leq A||f||_{L^{q}(B)}+2M_{0}
\]

\[
\leq AS_{0}+2M_{0}\coloneqq C(q,M_{0},S_{0}),
\]

where $A=A(q)$ is some constant. 
\end{onehalfspace}
\end{proof}
\begin{onehalfspace}
Since the mass and energy of H-solutions are constant in time and
the entropy is uniformly bounded in time we immediately get 
\end{onehalfspace}
\begin{cor}
\begin{onehalfspace}
\begin{flushleft}
\label{Corollary 3.3 } There are constants $0<c=c(M_{0},E_{0},H_{0},K),C=C(M_{0},S_{0},q)$
with the following property. Let $f$ be a H-solution with $M(f_{0})=M_{0},E(f_{0})=E_{0},H(f_{0})=H_{0},||f(t,\cdot)||_{L^{q}(B)}\leq S_{0}$
for some $\frac{3}{2}<q\leq\infty$. Let $K\Subset\mathbb{R}^{3}$.
Then for all $\xi\in\mathbb{R}^{n}$:$c|\xi|^{2}\leq\overline{a_{ij}}^{f}(t,v)\xi_{i}\xi_{j}\leq C|\xi|^{2}$
for all $(t,v)\in J\times K$. 
\par\end{flushleft}
\end{onehalfspace}

\end{cor}
\begin{onehalfspace}
The strategy of the proof of theorem \ref{Theorem 3.4} will be roughly
as follows.We start by obtaining regularity of the coefficients $\overline{a}_{ij}^{f}$
. Then we localize the solution in order to obtain a linear parabolic
PDE in divergence form. This will allow us to apply classical existence
and uniqueness results from the theory of linear parabolic PDE's.
The local ellipticity of the coefficients (Corollary \ref{Corollary 3.3 })
will be freely and frequently used in the sequel. We start by recalling
the following parabolic version of the celebrated De Giorgi-Nash-Moser
method
\end{onehalfspace}
\begin{thm}
\begin{onehalfspace}
\begin{flushleft}
\textup{(\cite{19}, Theorem 18)} \label{Theorem 3.5} Let $V:B\rightarrow\mathbb{R}^{3}$
be a vector field such that $|V|^{2}\in L^{p}L^{q}(\omega)$, where
$1<p<\infty,1<q<\infty$ satisfy $\frac{2}{p}+\frac{3}{q}<2$. Suppose
$u\in L^{\infty}L^{2}(\omega)$ is a weak subsolution to 
\par\end{flushleft}
\begin{flushleft}
\begin{equation}
\partial_{t}u-\partial_{j}(\overline{a}_{ij}^{f}\partial_{i}u)+\nabla u\cdot V\geq0\label{eq:-1-1}
\end{equation}
 Then there is some $\alpha>0$ such that $u\in H^{\alpha,\frac{\alpha}{2}}(\Omega)$
for all $\Omega\Subset\omega$.
\par\end{flushleft}
\end{onehalfspace}

\end{thm}
\begin{onehalfspace}
We will also heavily rely on the following Calderon Zygmund type theorem 
\end{onehalfspace}
\begin{thm}
\begin{onehalfspace}
\begin{flushleft}
\textup{(\cite{5}, Theorem 4.12)} \label{Theorem 3.6 } Suppose $\nu\in C^{\infty}(\mathbb{R}^{3}\setminus\{0\})$
has the form $\nu(y)=\mu(\frac{y}{|y|})$ where $\mu\in L^{q}(S^{2})$
for some $q>1$ is an even function such that $\underset{S^{2}}{\int}\mu(z)d\sigma(z)=0$.
Then for each $1<p<\infty$ the operator $T:L^{p}(\mathbb{R}^{3})\rightarrow L^{p}(\mathbb{R}^{3})$
defined by $Tf=(\frac{\nu(y)}{|y|^{3}}\ast f(y))(x)$ is bounded. 
\par\end{flushleft}
\end{onehalfspace}

\end{thm}
\begin{onehalfspace}
We shall first verify that the conditions imposed in theorem \ref{Theorem 3.6 }
are indeed verified for the second derivatives of $a_{ij}$. This
verification is based on elementary (yet somewhat tedious) calculations,
and is the content of the following 
\end{onehalfspace}
\begin{lem}
\begin{onehalfspace}
\begin{flushleft}
\label{Lemma 3.6-1} For each $1\leq k,l\leq3$ it holds that $\partial_{kl}a_{ij}(y)=\frac{\nu_{kl}(y)}{|y|^{3}}$
where $\nu_{kl}$ is as in theorem \ref{Theorem 3.6 }. 
\par\end{flushleft}
\end{onehalfspace}

\end{lem}
\begin{onehalfspace}
\textit{Proof. }We differentiate 

\[
\partial_{k}a_{ij}(z)=\partial_{k}(\frac{\delta_{ij}}{|z|}+\frac{z_{i}z_{j}}{|z|^{3}})=-\frac{\delta_{ij}z_{k}}{|z|^{3}}+\frac{\delta_{ik}z_{j}+\delta_{jk}z_{i}}{|z|^{3}}-\frac{3z_{i}z_{j}z_{k}}{|z|^{5}}=\frac{\delta_{ik}z_{j}+\delta_{jk}z_{i}-\delta_{ij}z_{k}}{|z|^{3}}-\frac{3z_{i}z_{j}z_{k}}{|z|^{5}}
\]
. 

\[
\partial_{l}(\frac{\delta_{ik}z_{j}+\delta_{jk}z_{i}-\delta_{ij}z_{k}}{|z|^{3}})=\frac{\delta_{ik}\delta_{lj}+\delta_{jk}\delta_{il}-\delta_{ij}\delta_{kl}}{|z|^{3}}-\frac{3z_{l}(\delta_{ik}z_{j}+\delta_{jk}z_{i}-\delta_{ij}z_{k})}{|z|^{5}}
\]

\[
\partial_{l}(\frac{3z_{i}z_{j}z_{k}}{|z|^{5}})=\frac{3\delta_{il}z_{j}z_{k}+3\delta_{lj}z_{i}z_{k}+3\delta_{kl}z_{j}z_{i}}{|z|^{5}}-\frac{15z_{i}z_{j}z_{k}z_{l}}{|z|^{7}}.
\]

Hence 

\[
\partial_{kl}a_{ij}(z)=\partial_{kl}(\frac{\delta_{ij}}{|z|}+\frac{z_{i}z_{j}}{|z|^{2}})=
\]

\[
\frac{\delta_{ik}\delta_{lj}+\delta_{jk}\delta_{il}-\delta_{ij}\delta_{kl}}{|z|^{3}}-\frac{3z_{l}(\delta_{ik}z_{j}+\delta_{jk}z_{i}-\delta_{ij}z_{k})}{|z|^{5}}-\frac{3\delta_{il}z_{j}z_{k}+3\delta_{lj}z_{i}z_{k}+3\delta_{kl}z_{j}z_{i}}{|z|^{5}}+\frac{15z_{i}z_{j}z_{k}z_{l}}{|z|^{7}}=
\]

\[
=\frac{1}{|z|^{3}}(\delta_{ik}\delta_{lj}+\delta_{jk}\delta_{il}-\delta_{ij}\delta_{kl}-\frac{3z_{l}(\delta_{ik}z_{j}+\delta_{jk}z_{i}-\delta_{ij}z_{k})+3\delta_{il}z_{j}z_{k}+3\delta_{lj}z_{i}z_{k}+3\delta_{kl}z_{j}z_{i}}{|z|^{2}}+\frac{15z_{i}z_{j}z_{k}z_{l}}{|z|^{4}}).
\]

Denote by $P_{i},1\leq i\leq3$ the projection on the $i$-th coordinate.
With this notation we recognize the following identity 

\[
\partial_{kl}(a_{ij})=\frac{1}{|z|^{3}}(\delta_{ik}\delta_{lj}+\delta_{jk}\delta_{il}-\delta_{ij}\delta_{kl}-3\delta_{ik}P_{l}(\frac{z}{|z|})P_{j}(\frac{z}{|z|})-3\delta_{jk}P_{i}(\frac{z}{|z|})P_{l}(\frac{z}{|z|})+3\delta_{ij}P_{l}(\frac{z}{|z|})P_{k}(\frac{z}{|z|})
\]

\[
-3\delta_{il}P_{j}(\frac{z}{|z|})P_{k}(\frac{z}{|z|})-3\delta_{lj}P_{i}(\frac{z}{|z|})P_{k}(\frac{z}{|z|})-3\delta_{kl}P_{j}(\frac{z}{|z|})P_{i}(\frac{z}{|z|})+15P_{i}(\frac{z}{|z|})P_{j}(\frac{z}{|z|})P_{l}(\frac{z}{|z|})P_{k}(\frac{z}{|z|}))\coloneqq\frac{\mu_{kl}(\frac{z}{|z|})}{|z|^{3}}.
\]

It is apparent that $\mu_{kl}$ is even and $\mu_{kl}\in L^{q}(S^{2})$
for $q>1$. In addition, the following calculations are an elementary
exercise in calculus (see e.g. \cite{7})

\[
\underset{S^{2}}{\int}\delta_{ik}\delta_{lj}d\sigma(z)=4\pi\delta_{ik}\delta_{lj}.
\]

\[
\underset{S^{2}}{\int}P_{i}(z)P_{j}(z)d\sigma(z)=\frac{4\pi}{3}\delta_{ij}.
\]

\[
\underset{S^{2}}{\int}P_{i}(z)P_{j}(z)P_{l}(z)P_{k}(z)d\sigma(z)=-\frac{4\pi}{15}(\delta_{ik}\delta_{lj}+\delta_{jk}\delta_{il}-\delta_{ij}\delta_{kl}-\delta_{ik}\delta_{lj}-\delta_{jk}\delta_{il}+\delta_{ij}\delta_{lk}-\delta_{il}\delta_{jk}-\delta_{lj}\delta_{ik}-\delta_{kl}\delta_{ij}).
\]

With the aid of the above identities it is readily checked that $\underset{S^{2}}{\int}\mu(z)d\sigma(z)=0$. 
\end{onehalfspace}
\begin{onehalfspace}
\begin{flushright}
$\square$
\par\end{flushright}
\end{onehalfspace}
\begin{lem}
\begin{onehalfspace}
\begin{flushleft}
\label{Lemma 3.6} Let $f$ be a H-solution with $||f(t,\cdot)||_{L^{q}(B)}\leq S_{0}$
where $3<q\leq\infty$. Then:
\par\end{flushleft}
\begin{flushleft}
1. There is some $\alpha>0$ such that $f\in H^{\alpha,\frac{\alpha}{2}}(\Omega)$
for all $\Omega\Subset\omega$. 
\par\end{flushleft}
\begin{flushleft}
2. (i) $a_{ij}\in C^{0}(\omega)$.
\par\end{flushleft}
\begin{flushleft}
2. (ii) For each $t\in J$ the function $a_{ij}(t,\cdot)$ is differentiable
on $B$ and $\partial_{k}a_{ij}\in C^{0}(\omega)$. 
\par\end{flushleft}
\end{onehalfspace}
\end{lem}
\begin{proof}
\begin{onehalfspace}
We assume with out loss of generality $q<\infty$. We start by showing
$\overline{b_{i}}^{f}\in L_{\mathrm{loc}}^{\infty}(\omega)$. 

Let $B'\Subset B$ be a ball and $J'\Subset J$. Pick $B''\Subset B$
to be a ball such that $B'\Subset B''$ and any $2\epsilon-$neighborhood
of $B'$ is $\Subset B''$, for some sufficiently small $\epsilon>0$.
Pick $\chi\in C_{0}^{\infty}(\mathbb{\mathbb{R}}^{3})$ with $\chi(z)\equiv1$
on $|z|\leq\epsilon$, $\chi\equiv0$ on $|z|>2\epsilon$ . Then for
all $(t,v)\in J'\times B'$ we have 
\[
|\overline{b_{i}}^{f}(t,v)|\lesssim\underset{\mathbb{R}^{3}}{\int}\frac{1}{|v-w|^{2}}f(t,w)dw=\underset{\mathbb{R}^{3}}{\int}\frac{\chi(v-w)}{|v-w|^{2}}f(t,w)dw+\underset{\mathbb{R}^{3}}{\int}\frac{1-\chi(v-w)}{|v-w|^{2}}f(t,w)dw=
\]

\[
\underset{|v-w|\leq2\epsilon}{\int}\frac{\chi(v-w)}{|v-w|^{2}}f(t,w)dw+\underset{\mathbb{R}^{3}}{\int}\frac{1-\chi(v-w)}{|v-w|^{2}}f(t,w)dw\leq\underset{B''}{\int}\frac{\chi(v-w)}{|v-w|^{2}}f(t,w)dw+||\frac{1-\chi(w)}{|w|^{2}}||_{\infty}||f(t,\cdot)||_{L^{1}(\mathbb{R}^{3})}
\]

\[
\lesssim||f(t,\cdot)||_{L^{q}(B'')}+||f(t,\cdot)||_{L^{1}(\mathbb{R}^{3})}.
\]

where the last inequality follows from H\"older nequality and the
assumption $q>3$. Taking the double supremum on both sides gives
$\overline{b_{i}}^{f}\in L_{\mathrm{loc}}^{\infty}(\omega)$. 

1. Now we show that $f$ is locally H\"older continuous in $\omega$.
To this aim we wish to show that $f$ (which by assumption is $L^{\infty}L^{q}(\omega),q>3$)
is a subsolution to an inequality of the form (\ref{eq:-1-1}). Let
$\omega'=(0,S')\times B'\Subset\omega$. Corollary \ref{Corollary 2.3}
and Remark \ref{Remark 2.4} imply in particular the following equation
for all $0\leq\varphi\in C_{0}^{2}(\omega')$
\begin{equation}
-\underset{\omega'}{\int}f\partial_{t}\varphi+\underset{\omega'}{\int}(\overline{a}_{ij}^{f}\partial_{i}f)(\partial_{j}\varphi)=\underset{\omega'}{\int}f\overline{b}_{i}^{f}\partial_{i}\varphi(t,v).\label{eq:-5}
\end{equation}

That $f\in W_{2}^{1,0}(\omega')$ is a particular byproduct of theorem
\ref{Theorem 2.3 }. Furthermore, in step 2.i we will prove (independently)
that $\partial_{kl}\overline{a_{ij}}^{f}(t,\cdot)\in L_{\mathrm{loc}}^{q}(B)$
for each fixed $t$, so that $\overline{b}_{i}^{f}(t,\cdot)\in W_{2}^{1}(B')$
for each fixed $t$. 

Therefore we may integrate by parts the RHS of equation (\ref{eq:-5})
and arrive at the equation 

\[
-\underset{\omega'}{\int}f\partial_{t}\varphi+\underset{\omega'}{\int}(\overline{a}_{ij}^{f}\partial_{i}f)(\partial_{j}\varphi)=-\underset{\omega'}{\int}(\overline{b}_{i}^{f}\partial_{i}f+f\partial_{i}\overline{b}_{i}^{f})\varphi.
\]

Hence 

\begin{equation}
-\underset{\omega'}{\int}f\partial_{t}\varphi+\underset{\omega'}{\int}(\overline{a}_{ij}^{f}\partial_{i}f)(\partial_{j}\varphi)+\underset{\omega'}{\int}\nabla f\cdot V\varphi=-\underset{\omega'}{\int}f\varphi\partial_{i}\overline{b}_{i}^{f}.\label{eq:-2}
\end{equation}

where $V=(\overline{b}_{1}^{f},\overline{b}_{2}^{f},\overline{b}_{3}^{f})$.
By an elementary calculation $-\partial_{i}(b_{i})$ is positively
proportional to the Dirac distribution, which in turn implies that
$-\partial_{i}(\overline{b}_{i}^{f})$ is positively proportional
to $f$. Therefore the RHS of (\ref{eq:-2}) is nonnegative. Thus 

\[
-\underset{\omega'}{\int}f\partial_{t}\varphi+\underset{\omega'}{\int}(\overline{a}_{ij}^{f}\partial_{i}f)(\partial_{j}\varphi)+\underset{\omega'}{\int}\nabla f\cdot V\varphi-\underset{B'}{\int}f\varphi(0,v)dv\geq0.
\]

Since $\overline{b_{i}}^{f}\in L_{\mathrm{loc}}^{\infty}(\omega)$
we know in particular that $|V|^{2}\in L^{\infty}L^{q}(\omega')$
for some $q>\frac{3}{2}$ which obviously implies the integrability
condition imposed for $V$ in theorem \ref{Theorem 3.5}. Therefore
$f$ is locally H\"older continuous in $\omega$.

2. i. We show that for each fixed $t$, $\overline{a_{ij}}^{f}(t,\cdot)$
is differentiable. 

Let $B'\Subset B,J'\Subset J$. Denote by $\rho>0$ the radius of
$B'$ and pick $\epsilon>0$ so small so that $B''\coloneqq B{}_{\rho+2\epsilon}\Subset B$. 

\[
\mathbf{1}_{B'}(v)(\partial_{kl}\overline{a_{ij}}^{f})(t,v)=\mathbf{1}_{B'}(v)\underset{\mathbb{R}^{3}}{\int}\partial_{kl}a_{ij}(v-w)f(t,w)dw=
\]

\[
=\mathbf{1}_{B'}(v)\underset{\mathbb{R}^{3}}{\int}\partial_{kl}a_{ij}(v-w)\mathbf{1}_{B''}(w)f(t,w)dw+\mathbf{1}_{B'}(v)\underset{\mathbb{R}^{3}}{\int}\partial_{kl}a_{ij}(v-w)\mathbf{1}_{\mathbb{R}^{3}-B''}(w)f(t,w)dw=
\]

\[
=\mathbf{1}_{B'}(v)\underset{\mathbb{R}^{3}}{\int}\partial_{kl}a_{ij}(v-w)\mathbf{1}_{B''}(w)f(t,w)dw+\mathbf{1}_{B'}(v)\underset{\mathbb{R}^{3}}{\int}\partial_{kl}a_{ij}(v-w)\mathbf{1}_{|v-w|\geq2\epsilon}\mathbf{1}_{\mathbb{R}^{3}-B''}(w)f(t,w)dw\coloneqq g_{t}(\cdot)+h_{t}(\cdot).
\]

Lemma \ref{Lemma 3.6-1}, theorem \ref{Theorem 3.6 } and the assumption
on $f$ imply that $g_{t}\in L^{q}(B')$. Furthermore, by Young's
convolution inequality: 

\[
||h_{t}(\cdot)||_{L^{q}(B')}=||h_{t}(\cdot)||_{L^{q}(\mathbb{R}^{3})}=||(|\mathbf{1}_{|\cdot|\geq2\epsilon}(\partial_{kl}a_{ij})|)\ast(\mathbf{1}_{\mathbb{R}^{3}-B''}f)||_{L^{q}(\mathbb{R}^{3})}\leq||\mathbf{1}_{|\cdot|\geq2\epsilon}\partial_{kl}a_{ij}||_{L^{q}(\mathbb{R}^{3})}||\mathbf{1}_{\mathbb{R}^{3}-B''}f(t,\cdot)||_{L^{1}(\mathbb{R}^{3})}<\infty.
\]

We move to show that $\overline{a_{ij}}^{f}(t,\cdot)\in L^{q}(\omega)$.
The proof is similar to the argument presented at the begining. Denote
by $R$ the radius of the ball $B.$ We have 

\[
\underset{\mathbb{R}^{3}}{\int}|a_{ij}(v-z)|f(t,z)dz\leq2\underset{\mathbb{R}^{3}}{\int}\frac{1}{|v-z|}f(t,z)dz=2\underset{\mathbb{R}^{3}}{\int}\frac{1}{|z|}f(t,v-z)dz=
\]

\[
2\underset{|z|\leq R}{\int}\frac{1}{|z|}f(t,v-z)dz+2\underset{|z|>R}{\int}\frac{1}{|z|}f(t,v-z)dz\leq C(||f(t,\cdot)||_{L^{q}(B)}+||f(t,\cdot)||_{1}).
\]
for some constant $C=C(q)$ (The finiteness of $\underset{|z|\leq R}{\int}\frac{1}{|z|^{\frac{q}{q-1}}}dz$
is guaranteed because of the assumption $q>3$). So $\underset{\omega}{\sup}\underset{\mathbb{R}^{3}}{\int}|a_{ij}(v-z)|f(t,z)dz\leq C(q)(||f||_{L^{\infty}L^{q}(\omega)}+||f||_{L^{\infty}L^{1}(\mathbb{R}^{3})})$,
and in particular $\overline{a_{ij}}^{f}(t,\cdot)\in L^{q}(\omega)$.
Thus we have proved $\overline{a_{ij}}^{f}(t,\cdot)\in W^{2,q}(B')$
for $q>3$ which by Sobolev embedding implies that $\overline{a_{ij}}^{f}(t,\cdot)$
is $C^{1}(B')$. 

We prove continuity with respect to $t$. Pick $\chi\in C_{0}^{\infty}(\mathbb{R}^{3})$
with $\chi\equiv1$ on $B_{\epsilon}(0)$ and $\mathrm{supp}(\chi)\Subset B_{2\epsilon}(0)$.
\[
\mathbf{1}_{B'}(v)\overline{a_{ij}}^{f}(t,v)=\mathbf{1}_{B'}\underset{\mathbb{R}^{3}}{\int}a_{ij}(v-w)f(t,w)dw=\mathbf{1}_{B'}(\underset{\mathbb{R}^{3}}{\int}\chi(v-w)a_{ij}(v-w)f(t,w)dw+\underset{\mathbb{R}^{3}}{\int}(1-\chi(v-w))a_{ij}(v-w)f(t,w)dw)=
\]

\[
=\mathbf{1}_{B'}(\underset{|v-w|\leq2\epsilon}{\int}\chi(v-w)a_{ij}(v-w)f(t,w)dw+\underset{\mathbb{R}^{3}}{\int}(1-\chi(v-w))a_{ij}(v-w)f(t,w)dw)=
\]

\[
=\mathbf{1}_{B'}\underset{|v-w|\leq2\epsilon}{\int}\chi(v-w)a_{ij}(v-w)\mathbf{1}_{B''}(w)f(t,w)dw+\mathbf{1}_{B'}\underset{|v-w|\leq2\epsilon}{\int}\chi(v-w)a_{ij}(v-w)\mathbf{1}_{\mathbb{R}^{3}\setminus B''}(w)f(t,w)dw
\]

\[
+\mathbf{1}_{B'}\underset{\mathbb{R}^{3}}{\int}(1-\chi(v-w))a_{ij}(v-w)f(t,w)dw=
\]

\[
\mathbf{1}_{B'}\underset{|v-w|\leq2\epsilon}{\int}\chi(v-w)a_{ij}(v-w)\mathbf{1}_{B''}(w)f(t,w)dw+\mathbf{1}_{B'}\underset{\mathbb{R}^{3}}{\int}(1-\chi(v-w))a_{ij}(v-w)f(t,w)dw\coloneqq\widetilde{g_{t}}(v)+\widetilde{h_{t}}(v).
\]

Keep $v\in B'$ fixed. That $t\mapsto\widetilde{g_{t}}(v)$ is continuous
is an immediate consequence of 1 and the CS inequality. Furthermore
it is clear that $w\mapsto(1-\chi(v-w))a_{ij}(v-w)\in W^{2,\infty}(\mathbb{R}^{3})$,
which by theorem \ref{Theorem 2.2} implies that $t\mapsto\widetilde{h_{t}}(v)$
is continuous. So $\overline{a_{ij}}^{f}(\cdot,v)$ is continuous
on $J$ as a sum of such functions. 

ii. We already know that $\partial_{k}a_{ij}(t,\cdot)$ is continuous
by 2.i. Continuity with respect to $t$ is achieved as in 2.i ( here
we use H\"older's inequality instead of CS). 
\end{onehalfspace}
\end{proof}
\begin{onehalfspace}
Before giving the proof of theorem \ref{Theorem 3.4}, we will need
the following existence and uniqueness results, which will also prove
themselves useful in section 4. 
\end{onehalfspace}
\begin{thm}
\begin{onehalfspace}
\begin{flushleft}
\textup{\label{Theorem 3.7} (Theorem 6.1,III in \cite{15})} Suppose
$A_{ij}(t,v)$ satisfies the following conditions:
\par\end{flushleft}
\begin{flushleft}
1. $A_{ij}$ are locally uniformly elliptic on $\omega$. 
\par\end{flushleft}
\begin{flushleft}
2. For all $t\in J,$ $A_{ij}(t,\cdot)$ are differentiable with respect
to $v$ and $\underset{J}{\mathrm{ess}\sup}|\partial_{v_{k}}A_{ij}|<\infty$
for all $1\leq k\leq3$. 
\par\end{flushleft}
\begin{flushleft}
Suppose $F\in L^{2}(\omega)$. Then the problem 
\begin{equation}
\left\{ \begin{array}{cc}
\partial_{t}u-\partial_{j}(A_{ij}\partial_{i}u)=F & \omega\\
u=0 & [0,S]\times\partial B\\
u=0 & \{0\}\times B
\end{array}\right.\label{eq:}
\end{equation}
\par\end{flushleft}
\begin{flushleft}
has a unique solution from $W_{2}^{2,1}(\omega)$. Moreover, this
solution is a strong solution. 
\par\end{flushleft}
\end{onehalfspace}

\end{thm}
\begin{onehalfspace}

\end{onehalfspace}\begin{thm}
\begin{onehalfspace}
\begin{flushleft}
\textup{\label{Theorem 3.8} (Theorem 3.3, III \cite{15})} Suppose
$A_{ij}(t,v)$ are locally uniformly elliptic. Let $F\in L^{2}(\omega)$.
Then problem (\ref{eq:}) cannot have more than one weak solution
in $W_{2}^{1,0}(\omega)$. 
\par\end{flushleft}
\end{onehalfspace}

\end{thm}
\begin{onehalfspace}
\textit{Proof of theorem} \ref{Theorem 3.4}. Suppose $\Omega=(T_{1},T_{2})\times B'\Subset\omega$.
Let $\omega'\coloneqq(T_{1}',T_{2}')\times B''$ such that $\Omega\Subset\omega'\Subset\omega$.
Let $\zeta\in C_{0}^{\infty}(\omega')$ such that $\zeta\equiv1$
on $\Omega$ and put $u=\zeta f$. Fix some test function $\chi\in C_{0}^{2}([0,T)\times\mathbb{R}^{3})$.
We compute 

\[
\underset{\omega'}{\int}-u\partial_{t}\chi+(\overline{a}_{ij}^{f})\partial_{i}u\partial_{j}\chi=\stackrel[0]{T}{\int}\underset{\mathbb{R}^{3}}{\int}-u\partial_{t}\chi+(\overline{a}_{ij}^{f})\partial_{i}u\partial_{j}\chi=\stackrel[0]{T}{\int}\underset{\mathbb{R}^{3}}{\int}-\zeta f\partial_{t}\chi+\stackrel[0]{T}{\int}\underset{\mathbb{R}^{3}}{\int}(\overline{a}_{ij}^{f})(\partial_{i}\zeta f+\partial_{i}f\zeta)\partial_{j}\chi
\]

\[
=\stackrel[0]{T}{\int}\underset{\mathbb{R}^{3}}{\int}-\zeta f\partial_{t}\chi+\stackrel[0]{T}{\int}\underset{\mathbb{R}^{3}}{\int}(\overline{a}_{ij}^{f})\partial_{i}f(\partial_{j}(\zeta\chi)-\chi\partial_{j}\zeta)+\stackrel[0]{T}{\int}\underset{\mathbb{R}^{3}}{\int}(\overline{a}_{ij}^{f})f\partial_{j}\chi\partial_{i}\zeta=
\]
 
\[
=-\stackrel[0]{T}{\int}\underset{\mathbb{R}^{3}}{\int}f\partial_{t}(\zeta\chi)+\stackrel[0]{T}{\int}\underset{\mathbb{R}^{3}}{\int}(\overline{a}_{ij}^{f})\partial_{i}f\partial_{j}(\zeta\chi)+\stackrel[0]{T}{\int}\underset{\mathbb{R}^{3}}{\int}f\chi\partial_{t}\zeta-\stackrel[0]{T}{\int}\underset{\mathbb{R}^{3}}{\int}(\overline{a}_{ij}^{f})\chi\partial_{i}f\partial_{j}\zeta+\stackrel[0]{T}{\int}\underset{\mathbb{R}^{3}}{\int}(\overline{a}_{ij}^{f})f\partial_{j}\chi\partial_{i}\zeta.
\]

By equation (\ref{eq:-3}) the last sum is 

\[
=\stackrel[0]{T}{\int}\underset{\mathbb{R}^{3}}{\int}f\overline{b}_{i}^{f}\partial_{i}(\zeta\chi)+\stackrel[0]{T}{\int}\underset{\mathbb{R}^{3}}{\int}(\partial_{t}\zeta)f\chi-\stackrel[0]{T}{\int}\underset{\mathbb{R}^{3}}{\int}(\overline{a}_{ij}^{f})\chi\partial_{i}f\partial_{j}\zeta+\stackrel[0]{T}{\int}\underset{\mathbb{R}^{3}}{\int}(\overline{a}_{ij}^{f})f\partial_{j}\chi\partial_{i}\zeta.
\]

Since $f\in W_{2}^{1,0}(\omega')$ and $\overline{a_{ij}}^{f}(t,\cdot),\overline{b}_{i}^{f}(t,\cdot)\in W_{2}^{1}(B')$
for each fixed $t$, the last expression may be integrated by parts
and recasted as

\[
=-\underset{\omega'}{\int}(\zeta f\partial_{i}(\overline{b}_{i}^{f})+\zeta\overline{b}_{i}^{f}\partial_{i}f)\chi+\underset{\omega}{\int}\chi(f\partial_{t}\zeta-f\partial_{j}(\overline{a}_{ij}^{f})\partial_{i}\zeta-\overline{a}_{ij}^{f}(f\partial_{i}\partial_{j}\zeta+\partial_{j}f\partial_{i}\zeta)-\overline{a}_{ij}^{f}\partial_{i}f_{i}\partial_{j}\zeta)=
\]

\begin{equation}
=\underset{\omega'}{\int}(f\partial_{t}\zeta-f\partial_{j}(\overline{a}_{ij}^{f})\partial_{i}\zeta-\overline{a}_{ij}^{f}(f\partial_{i}\partial_{j}\zeta+\partial_{j}f\partial_{i}\zeta)-\overline{a}_{ij}^{f}\partial_{i}f\partial_{j}\zeta-\zeta f\partial_{i}(\overline{b}_{i}^{f})-\zeta\overline{b}_{i}^{f}\partial_{i}f)\chi\coloneqq\underset{\omega'}{\int}F\chi.\label{eq:5-1}
\end{equation}

Thus, we find that $u$ is a $W_{2}^{1,0}(\omega')$ solution to the
linear equation 

\begin{equation}
\left\{ \begin{array}{cc}
\partial_{t}u-\partial_{j}(\overline{a}_{ij}^{f}\partial_{i}u)=F & \omega'\\
u=0 & [T_{1}',T_{2}']\times\partial B''\\
u=0 & \{T_{1}'\}\times B''
\end{array}\right..\label{eq:5}
\end{equation}

Keeping in mind Step 1 of lemma \ref{Lemma 3.6} (in particular $f\in L_{loc}^{\infty}(\omega)$
and $f\in L^{\infty}L^{s'}(\omega)$ where $s'>3$), we make the following
key observations

i. $f\in L^{2}(\omega')$ 

ii. By theorem \ref{Theorem 2.3 } $(\sqrt{f})_{i}\in L^{2}(\omega')$
and thus $f_{i}=2\sqrt{f}(\sqrt{f})_{i}\in L^{2}(\omega')$ 

iii. $\overline{a}_{ij}^{f},\overline{b}_{i}^{f}\in L^{\infty}(\omega')$ 

iv. $(\overline{b}_{i}^{f})_{i}f$ is proportional to $f^{2}$. 

From i-iv we immediately see that $F\in L^{2}(\omega')$. Moreover,
lemma \ref{Lemma 3.6} shows that the condition $\underset{(T_{1}',T_{2}')}{\mathrm{ess}\sup}|(\overline{a}_{ij}^{f})_{k}|<\infty$
required in theorem \ref{Theorem 3.7} is verified, and so we know
that equation (\ref{eq:5}) has a unique $W_{2}^{2,1}(\omega')$ solution,
call it $\widetilde{u}$. In particular $\widetilde{u}$ is a $W_{2}^{1,0}(\omega')$
solution. Now, viewing equation (\ref{eq:5}) as an equation for the
space $W_{2}^{1,0}(\omega')$ and owing to the uniqueness provided
by theorem \ref{Theorem 3.8}, it follows that $u=\widetilde{u}\in W_{2}^{2,1}(\omega')$,
which implies $f\in W_{2}^{2,1}(\Omega)$. 
\end{onehalfspace}
\begin{onehalfspace}
\begin{flushright}
$\square$
\par\end{flushright}
\end{onehalfspace}

\begin{onehalfspace}
With the aid of the following estimate we can improve the Lebesgue
exponent of the first order spatial derivatives of $f$ 
\end{onehalfspace}
\begin{lem}
\begin{onehalfspace}
\begin{flushleft}
\textup{(Lemma 3.3, II in \cite{15})} \label{Lemma 3.9 } Suppose
$u\in W_{q}^{2,1}(\omega)$ and $1\leq q\leq r<\infty,\frac{1}{q}-\frac{1}{r}\leq\frac{1}{5}$.
Then 
\par\end{flushleft}
\begin{flushleft}
\begin{equation}
||\partial_{v_{i}}u||_{r,\omega}\leq C_{1}(||u||_{q,\omega}+||\partial_{v_{i}}u||_{q,\omega}+||\partial_{v_{i}v_{j}}u||_{q,\omega})+C_{2}||u||_{q,\omega}.\label{eq:6-1}
\end{equation}
\par\end{flushleft}
\end{onehalfspace}
\end{lem}
\begin{cor}
\begin{onehalfspace}
\begin{flushleft}
\label{Corollary 3.10 } Let $f$ be a H-solution. Suppose there exist
$S_{0}>0,q>3$ such that for all $t\in J$ one has $||f(t,\cdot)||_{L^{q}(B)}\leq S_{0}$
.Then $\partial_{v_{i}}f\in L^{\frac{10}{3}}(\Omega)$ for all $\Omega\Subset\omega$. 
\par\end{flushleft}
\end{onehalfspace}
\end{cor}
\begin{proof}
\begin{onehalfspace}
By Theorem \ref{Theorem 3.4} $f\in W_{2}^{2,1}(\Omega)$ and so taking
$q=2$ in Lemma \ref{Lemma 3.9 } we find that $\partial_{v_{i}}f\in L^{r}(\Omega)$
as long as $1\leq r\leq\frac{10}{3}$. 
\end{onehalfspace}
\end{proof}
\begin{onehalfspace}
In the next section we will iterate lemma \ref{Lemma 3.9 } in order
to show that $f$ lies locally in $W_{r}^{2,1}(\Omega)$ for arbitrary
$1\leq r<\infty$. 
\end{onehalfspace}
\begin{onehalfspace}

\section{From $\mathbf{\mathbf{W}_{2}^{2,1}}$ to $\mathbf{H_{\alpha+2}^{\ast}}$}
\end{onehalfspace}

\begin{onehalfspace}
We now wish to push further the main result obtained in the previous
section, by proving that $f$ is locally $H_{\alpha+2}^{\ast}$, which
in particular implies that $f$ is a classical solution to equation
(\ref{eq:1-1}). To finish the proof of theorem \ref{Theorem 1.1 }
we will need 
\end{onehalfspace}
\begin{thm}
\begin{onehalfspace}
\begin{flushleft}
\label{Theorem 4.5}\textup{ (\cite{15}, IV, Theorem 9.1)} Suppose
that $A_{ij}$ are locally uniformly elliptic and bounded continuous
on $\omega$. Suppose $\mathcal{F}\in L^{q}(\omega),1<q<\infty$.
Then the problem 
\par\end{flushleft}
\begin{flushleft}
\begin{equation}
\left\{ \begin{array}{cc}
\partial_{t}U-A_{ij}\partial_{i}\partial_{j}U=\mathcal{F} & \omega\\
U=0 & [T_{1},T_{2}]\times\partial B\\
U=0 & \{T_{1}\}\times B
\end{array}\right.\label{eq:7}
\end{equation}
\par\end{flushleft}
\begin{flushleft}
has a unique solution $U\in W_{q}^{2,1}(\omega)$. 
\par\end{flushleft}
\end{onehalfspace}
\end{thm}
\begin{lem}
\begin{onehalfspace}
\begin{flushleft}
\textup{(\cite{15}, Page 343)} \label{Lemma 4.5 } Let $5<q<\infty,$
$0<\alpha<1-\frac{5}{q}$ and let the conditions of theorem \ref{Theorem 4.5}
hold. The solution $u\in W_{q}^{2,1}(\omega)$ of equation (\ref{eq:7})
has $u\in H^{1+\alpha,\frac{1+\alpha}{2}}(\omega)$. 
\par\end{flushleft}
\end{onehalfspace}

\end{lem}
\begin{onehalfspace}
In addition we need to slightly refine the conclusion of 2.i of lemma
\ref{Lemma 3.6} by showing that $\overline{a_{ij}}^{f}$ are H\"older
continuous in space uniformly with respect to time. 
\end{onehalfspace}
\begin{lem}
\begin{onehalfspace}
\begin{flushleft}
\label{Lemma 4.4} Suppose $f\in H^{\alpha,\frac{\alpha}{2}}(\omega)$.
Let $\Omega\Subset\omega$. Then there is some $C>0$ such that for
all $(t,v_{1}),(t,v_{2})\in\Omega$ it holds that $|\overline{a_{ij}}^{f}(t,v_{1})-\overline{a_{ij}}^{f}(t,v_{2})|\leq C|v_{1}-v_{2}|^{\alpha}$.
In particular $\overline{a_{ij}}^{f}\in H_{\alpha}^{\ast}(\overline{\Omega})$. 
\par\end{flushleft}
\end{onehalfspace}
\end{lem}
\begin{proof}
\begin{onehalfspace}
With the same notation of lemma \ref{Lemma 3.6} we have 

\[
\mathbf{1}_{B'}(v)\overline{a_{ij}}^{f}(t,v)=\underset{B''}{\int}\chi(v-w)a_{ij}(v-w)f(t,w)dw+\mathbf{1}_{B'}\underset{\mathbb{R}^{3}}{\int}(1-\chi(v-w))a_{ij}(v-w)f(t,w)dw\coloneqq g_{t}(\cdot)+h_{t}(\cdot).
\]

Let $B_{v}''$ be the ball centered at $v$ with the same radius as
$B''$. 

\[
|g_{t}(v_{1})-g_{t}(v_{2})|\leq\underset{B_{v}''}{\int}|\chi(w)a_{ij}(w)||f(t,v_{1}-w)-f(t,v_{2}-w)|dw\leq C_{1}|v_{1}-v_{2}|^{\alpha}.
\]

In addition by the mean value theorem 

\[
|h_{t}(v_{1})-h_{t}(v_{2})|\leq||\nabla h_{t}(\xi)||\times|v_{1}-v_{2}|\leq C_{2}|v_{1}-v_{2}|,
\]

where $\xi$ is an intermidiate point and the second inequality is
because $||\nabla h_{t}||_{\infty}$ is easily seen to be bounded
indepnedently of $t$.
\end{onehalfspace}
\end{proof}
\begin{onehalfspace}
At the last step of the proof we will apply the following parabolic
version of interior Schauder estimates. 
\end{onehalfspace}
\begin{thm}
\begin{onehalfspace}
\begin{flushleft}
\label{Theorem 4.6 } \textup{(\cite{14}, Theorem 1)} Suppose:
\par\end{flushleft}
\begin{flushleft}
1. There are constants $\nu,\mu>0$ such that for all $\xi\in\mathbb{R}^{3}$
it holds that $\nu|\xi|^{2}\leq A_{ij}(t,v)\xi_{i}\xi_{j}\leq\mu|\xi|^{2}$
for all $(t,v)\in\omega$. 
\par\end{flushleft}
\begin{flushleft}
2. There is a constant $\kappa>0$ such that $||A_{ij}||_{H_{\alpha}^{\ast}(\overline{\omega})}\leq\kappa$
and $||d^{2}\mathcal{F}||_{\alpha}^{\ast}<\infty$. 
\par\end{flushleft}
\begin{flushleft}
If $u\in H_{\alpha+2}^{\ast}(\overline{\omega})$ is a solution to
the equation 
\par\end{flushleft}
\begin{flushleft}
\[
\partial_{t}u-A_{ij}\partial_{i}\partial_{j}u=\mathcal{F},
\]
\par\end{flushleft}
\begin{flushleft}
then $||u||_{H_{\alpha+2}^{\ast}(\overline{\omega})}\leq c(||d^{2}\mathcal{F}||_{H_{\alpha}^{\ast}(\overline{\omega})}+||u||_{0})$
where $c=c(\Omega,\kappa,\nu,\alpha)$. 
\par\end{flushleft}
\end{onehalfspace}

\end{thm}
\begin{onehalfspace}
Using the above Schauder estimates we can obtain the following uniqueness
and existence result. The proof is based on a standard method of continuity
argument. We include it here only for the sake of completeness, since
it does not appear explicitly in the literature.

\end{onehalfspace}\begin{cor}
\begin{onehalfspace}
\begin{flushleft}
\label{Corollary 4.5} Let the conditions 1+2 of theorem \ref{Theorem 4.6 }
hold. Write $\omega=(T_{1},T_{2})\times B$. The equation 
\par\end{flushleft}
\begin{flushleft}
\begin{equation}
\left\{ \begin{array}{cc}
\partial_{t}u-A_{ij}\partial_{i}\partial_{j}u=\mathcal{F} & (T_{1},T_{2})\times B\\
u=0 & [T_{1},T_{2}]\times\partial B\\
u=0 & \{T_{1}\}\times B
\end{array}\right.\label{eq:-1}
\end{equation}
\par\end{flushleft}
\begin{flushleft}
has a unique $H_{\alpha+2}^{\ast}(\overline{\omega})$ solution. 
\par\end{flushleft}
\end{onehalfspace}
\end{cor}
\begin{proof}
\begin{onehalfspace}
Denote $L=\partial_{t}u-\overline{a}_{ij}^{f}\partial_{i}\partial_{j}u$.
Consider the Banach spaces $\mathfrak{B}_{1}=H_{\alpha+2}^{\ast}(\overline{\omega})\cap\{u=0$
on $[T_{1},T_{2}]\times\partial B\cup\{T_{1}\}\times B\}$ and $\mathfrak{B}_{2}=H_{\alpha}^{\ast}(\overline{\omega})$.
For each $0\leq s\leq1$ consider the operator $L_{s}:\mathfrak{B}_{1}\rightarrow\mathfrak{B}_{2}$
defined by $L_{s}u=sLu+(1-s)(\partial_{t}u-\Delta u)$. The solvability
of equation (\ref{eq:-1}) for arbitrary $\mathcal{F}\in H_{\alpha}^{\ast}(\overline{\omega})$
is equivalent to the fact that $L_{1}$ is onto. Let $u\in\mathfrak{B}_{1}$
and write $L_{s}u=\mathscr{\mathcal{F}}\in H_{\alpha}^{\ast}(\overline{\omega})$
. By theorem \ref{Theorem 4.6 } 

we have 
\[
||u||_{\mathfrak{B}_{1}}=||u||_{H_{\alpha+2}^{\ast}(\overline{\omega})}\leq c(||d^{2}\mathcal{F}||_{\alpha}^{\ast}+||u||_{0})\leq c(||\mathcal{F}||_{H_{\alpha}^{\ast}(\overline{\omega})}+\underset{\omega}{\sup}|\mathcal{F}|)\leq c||\mathcal{F}||_{H_{\alpha}^{\ast}(\overline{\omega})}=c||L_{s}u||_{H_{\alpha}^{\ast}(\overline{\omega})},
\]
 where the second inequality is by the maximum principle and $c$
is independent of $s$. Indeed, note that if $\nu$ stands for the
lower ellipticity constant of $L$, then for each $s$ the lower ellipticity
constant of $L_{s}$ can be taken to be $\nu_{s}=\min(1,\nu)$, which
is independent of $s$. It is also apparent that the $H_{\alpha}^{\ast}(\overline{\omega})$
norm of the coefficients of the second order derivatives of $L_{s}$
is bounded independently of $s$. Therefore the constant $c$ from
theorem \ref{Theorem 4.6 } can be taken to be the same for all the
$L_{s}$. Furthermore, it is classical that the heat equation 

\[
\left\{ \begin{array}{cc}
\partial_{t}u-\Delta u=\mathcal{F} & (T_{1},T_{2})\times B\\
u=0 & [T_{1},T_{2}]\times\partial B\\
u=0 & \{T_{1}\}\times B
\end{array}\right.
\]

has a $H_{\alpha+2}^{\ast}(\overline{\omega})$ solution provided
$\mathcal{F}\in H_{\alpha}^{\ast}(\overline{\omega})$ (in fact less
regularity on $\mathcal{F}$ is required here). Otherwise put, $L_{0}$
is onto $\mathfrak{B}_{2}$. By the method of continuity (see e.g.
theorem 5.2 in \cite{9}) it follows that $L_{1}$ is onto, as desired.
Uniqueness of the solution is a consequence of the maximum principle. 
\end{onehalfspace}
\end{proof}
\begin{onehalfspace}
\textit{Proof of theorem }\ref{Theorem 1.1 }. 1. Given a cylinder
$(T_{1},T_{2})\times B'\Subset\omega$ pick a cylinder $(T_{1},T_{2})\times B'\Subset(T_{1}',T_{2}')\times B''\Subset\omega$.
Set $\omega'=(T_{1}',T_{2}')\times B''$. We localize the solution
$f$ by defining $u=\zeta f$ for some $\zeta\in C_{0}^{\infty}((T_{1}',T_{2}')\times B'')$
satisfying $\zeta\equiv1$ on $(T_{1},T_{2})\times B'$. First we
wish to show that $u$ satisfy an equation of the form (\ref{eq:7}).
As in theorem \ref{Theorem 3.4} we get that $u$ satisfies the equation 

\begin{equation}
\left\{ \begin{array}{cc}
\partial_{t}u-\overline{a}_{ij}^{f}\partial_{i}\partial_{j}u=\mathcal{F} & \omega'\\
u=0 & [T_{1}',T_{2}']\times\partial B''\\
u=0 & \{T_{1}'\}\times B''
\end{array}\right.\label{eq:E}
\end{equation}
where $\mathcal{F}=\zeta_{t}f-\overline{a}_{ij}^{f}(\zeta_{ij}f+f_{j}\zeta_{i})-\overline{a}_{ij}^{f}f_{i}\zeta_{j}-\zeta(\overline{b}_{i}^{f})_{i}f$.
By corollary \ref{Corollary 3.10 } we see that $\mathcal{F}\in L^{\frac{10}{3}}(\omega')$
(recall that $(\overline{b}_{i}^{f})_{i}$ is propotional to $f$)
and so by theorem \ref{Theorem 4.5} $u\in W_{\frac{10}{3}}^{2,1}(\omega')$.
We proceed by iterating lemma \ref{Lemma 3.9 }: utilizing lemma \ref{Lemma 3.9 }
with $q=\frac{10}{3},r=10$ we find that $u\in W_{10}^{2,1}(\omega')$.
Utilizing lemma \ref{Lemma 3.9 } once again with $q=10,10\leq r<\infty$
we get $u\in W_{r}^{2,1}(\omega')$. Thus, we have shown $u\in W_{r}^{2,1}(\omega')$
for all $1\leq r<\infty$. Lemma \ref{Lemma 4.5 } entails that $u\in H^{1+\alpha,\frac{1+\alpha}{2}}(\omega')$
for all $0<\alpha<1$. We have thus shown that $f\in H^{1+\alpha,\frac{1+\alpha}{2}}(\Omega)$
for any $\Omega\Subset\omega$ and $0<\alpha<1$, which in turn implies
that $\mathcal{F}$ verifies the H\"older condition 2 in theorem
\ref{Theorem 4.6 } for all $0<\alpha<1$. In addition lemma \ref{Lemma 4.4}
guarantees that $\overline{a}_{ij}^{f}$ verifies the H\"older condition
2 imposed in Theorem \ref{Theorem 4.6 } for all $0<\alpha<1$. Corollary
\ref{Corollary 4.5} ensures that equation (\ref{eq:E}) has a unique
$H_{\alpha+2}^{\ast}(\overline{\omega'})$ solution (for arbitrary
fixed $0<\alpha<1$) , and a fortiori this solution must identify
with $f$. 
\end{onehalfspace}
\begin{onehalfspace}
\begin{flushright}
$\square$ 
\par\end{flushright}
\end{onehalfspace}

\begin{onehalfspace}
\textbf{Acknowledgement. }

This work is part of the author's PhD thesis. I would like to express
my deepest graditute towards my supervisor François Golse for exposing
me to this research direction, for many fruitful discussions and for
carefully reading previous versions of this manuscript and providing
insightful comments and improvements. I would also like to thank an
anonymous referee for providing comments which improved the quality
of this work. 
\end{onehalfspace}


\begin{thebibliography}{10}
\bibitem[1]{1} R. Alexandre, J. Liao and C. Lin\textit{. Some a priori
estimates for the homogeneous Landau equation with soft potentials}.
Kinet. Relat. Models \textbf{8} (2015), no. 4, 617--6. 

\bibitem[2]{2} A. A. Arsen\textquoteright ev, N. V. Peskov. \textit{The
existence of a generalized solution of Landau\textquoteright s equation}.
(Russian) Z. Vycisl. Mat i Mat. Fiz. \textbf{17} (1977), no. 4, 1063--1068,
1096.

\bibitem[3]{3} L. Desvillettes and C. Villani. \textit{On the spatially
homogeneous Landau equation for hard potentials part I. }Comm. in
Partial Diff. Eq. \textbf{25} (2000). 179--259.

\bibitem[4]{4} L. Desvillettes. \textit{Entropy dissipation estimates
for the Landau equation}. J. Funct. Anal. \textbf{269} (2015), 1359--1403. 

\bibitem[5]{5} J. Duoandikoetxea. \textit{Fourier analysis}. American
Mathematical Society. 29 (2000). 

\bibitem[6]{6} L.C. Evans. \textit{Partial differential equations.
}American Mathematical Society (2009). 

\bibitem[7]{7} G. B. Folland. \textit{How to integrate a polynomial
over a sphere. }The American Mathematical Monthly. \textbf{108} (2001).
446--448.

\bibitem[8]{8} N. Fournier. \textit{Uniqueness of bounded solutions
for the homogeneous Landau equation with a Coulomb potential}. Comm.
Math. Phys.\textbf{ 299} (2010), no. 3, 765--782.

\bibitem[9]{9} D. Gilbarg and N. Trudinger,\textit{ Elliptic partial
differential equations of second order.} Springer (1998).\textit{ }

\bibitem[10]{10} F. Golse, M. Gualdani, C. Imbert and A. Vasseur.
\textit{Partial regularity in time for the space homogeneous Landau
equation with Coulomb potential} (2019). arXiv:1906.0284. 

\bibitem[11]{11} F.Golse, C. Imbert, C. Mouhot and A. F. Vasseur.
\textit{Harnack inequality for kinetic Fokker-Planck equations with
rough coefficients and application to the Landau equation}. Ann. Sc.
Norm. Super. Pisa Cl. Sci. (5) 19 (2019), no. 1, 253--295. 

\bibitem[12]{12} M. Gualdani and N. Guillen. \textit{Estimates for
radial solutions of the homogeneous Landau equation with Coulomb potential}.
Anal. PDE \textbf{9} (2016), no. 8, 1772--1809. 

\bibitem[13]{13} M. Gualdani and N. Guillen. \textit{On Ap weights
and the Landau equation}. Calc. Var. Partial Differential Equations
\textbf{58} (2019), no. 1, Paper No. 17, 55 pp. 

\bibitem[14]{14} B. F. Knerr\textit{. Parabolic interior Schauder
estimates by the maximum principle. }Arch. Rational Mech. Anal. \textbf{75}
(1980). 51--58. 

\bibitem[15]{15} O.A. Ladyzhenskaya, V.A. Solonnikov and N.N. Uraltseva.\textit{
Linear and quasi-linear equations of parabolic type. }American Mathematical
Society (1968). 

\bibitem[16]{16} J. C. Robinson, J. L. Rodrigo and W. Sadowski. \textit{The
three dimensional Navier-Stokes equations. }Cambridge studies in advanced
mathematics (2016). 

\bibitem[17]{17} J. Serrin.\textit{ On the interior regularity of
weak solutions of the Navier-Stokes equations. }Arch. Rational Mech.
Anal. \textbf{9} (1962), 187--195. 

\bibitem[18]{18} L. Silvestre. \textit{Upper bounds for parabolic
eqyations and the Landau equation}. J. Differential Equations \textbf{262
}(2017), 3034--3055. 

\bibitem[19]{19} A. Vasseur. \textit{The De Giorgi method for elliptic
and parabolic equations and some applications}. Lectures on the analysis
of nonlinear partial differential equations. Part 4, 195--222, Morningside
Lect. Math., 4, Int. Press, Somerville, MA (2016).

\bibitem[20]{20} C. Villani. \textit{On a new class of weak solution
to the spatially homogeneous Boltzmann and Landau equations. }Arch.
Rational Mech. Anal. \textbf{143} (1998). 273--307.

\bibitem[21]{21} F. Weissler. \textit{Local existence and nonexistence
for semilinear parabolic equations in $L^{p}$. }Indiana Univ. Math.
J. \textbf{29} (1980), 79--102\textit{. }
\end{thebibliography}
\end{document}